\documentclass[11pt]{amsart}

\usepackage[dvipsnames]{xcolor}
\usepackage[title]{appendix}
\usepackage{tikz-cd}
\usepackage{changepage}
\usepackage{graphicx}
\usepackage{hyperref}
\usepackage{dsfont}
\usepackage{xfrac,faktor}
\usepackage{amsmath}
\usepackage{color}
\usepackage{caption}
\usepackage{subcaption}
\usepackage[T1]{fontenc}
\usepackage{mathtools}
\usepackage{amsfonts,amssymb,amscd,amsthm}
\usepackage{tikz}
\usepackage{comment}
\usepackage{float,url}
\usepackage{tikz-cd}
\usepackage[all]{xy}

\newlength{\mywidth}

\calclayout

\makeatletter
\DeclareFontFamily{U}{tipa}{}
\DeclareFontShape{U}{tipa}{m}{n}{<->tipa10}{}
\newcommand{\arc@char}{{\usefont{U}{tipa}{m}{n}\symbol{62}}}%

\newcommand{\arc}[1]{\mathpalette\arc@arc{#1}}

\newcommand{\arc@arc}[2]{%
  \sbox0{$\m@th#1#2$}%
  \vbox{
    \hbox{\resizebox{\wd0}{\height}{\arc@char}}
    \nointerlineskip
    \box0
  }%
}
\makeatother

\newtheorem{thm}{Theorem}[section]
\newtheorem{prop}[thm]{Proposition}

\newtheorem{cor}[thm]{Corollary}

\newtheorem*{cor*}{Corollary}

\theoremstyle{definition}
\newtheorem{definition}[thm]{Definition}

\newtheorem{notation}[thm]{Notation}

\newtheorem{thmx}{Theorem}

\makeatletter
\renewcommand*\env@matrix[1][\arraystretch]{%
  \edef\arraystretch{#1}%
  \hskip -\arraycolsep
  \let\@ifnextchar\new@ifnextchar
  \array{*\c@MaxMatrixCols c}}
\makeatother

\theoremstyle{remark}
\newtheorem{rem}[thm]{Remark}

\theoremstyle{question}

\tikzset{
  mynode/.style={fill,circle,inner sep=1pt,outer sep=0pt}
}

\begin{document}


\title[Quadrature Domains and the Real Quadratic Family]{Quadrature Domains and the Real Quadratic Family}


\author[K. Lazebnik]{Kirill Lazebnik}

\date{\today}


\maketitle

\begin{abstract} We study several classes of holomorphic dynamical systems associated with quadrature domains. Our main result is that real-symmetric polynomials in the principal hyperbolic component of the Mandelbrot set can be conformally mated with a congruence subgroup of $\textrm{PSL}(2,\mathbb{Z})$, and that this conformal mating is the Schwarz function of a simply connected quadrature domain.



\end{abstract}


\setcounter{tocdepth}{1}

\tableofcontents



\section{Introduction}

Parallels in the study of Kleinian groups and of complex dynamics are generally termed as entries in \emph{Sullivan's dictionary}. We refer to Section 3.5 of \cite{BF14} for an overview. In this paper we are concerned mainly with the following question: when can we ``combine'' the dynamical plane of a quadratic polynomial with that of a Kleinian group, and what is the resulting object? Let us introduce some terminology in order to make this question more precise.

Associated with the family of quadratic polynomials $p_c(z):=z^2+c$ is the Mandelbrot set $\mathcal{M}$: defined as those $c\in\mathbb{C}$ for which $p_c^n(0)\not\rightarrow\infty$ as $n\rightarrow\infty$. We list \cite{L6} as a comprehensive reference. Here we will primarily be concerned with the principal hyperbolic component of the Mandelbrot set: namely those $c\in\mathbb{C}$ such that $p_c$ has a finite attracting fixed point. We will denote by $K(p_c)$ the filled Julia set of $p_c$, and by $\mathcal{J}(p_c)$ the Julia set of $p_c$.

The Kleinian group relevant to our study is the congruence subgroup $\Gamma$ of $\textrm{PSL}(2,\mathbb{Z})$ defined by generators $\alpha(z):=-1/z$ and $\beta(z):=z+2$. A fundamental domain $U$ for $\Gamma$ is shown in Figure \ref{fig:fundamental_dom_group}. Associated with $\Gamma$ we define the following function $\rho: \mathbb{H}\setminus U\rightarrow \mathbb{H}$:\begin{align}
 \rho(z) := \left\{
        \begin{array}{ll}
        \alpha(z) & \quad \textrm{ if } |z|<1, \\            
        \beta^{-1}(z) & \quad \textrm{ if } \textrm{Re}(z)>1, \\
        \beta(z) & \quad \textrm{ if } \textrm{Re}(z)<-1. \\        
        \end{array}
    \right.
     \end{align} We will discuss further the connection between $\Gamma$ and $\rho$ in Section \ref{mod_group_struct}. 
         
\begin{figure}
\centering
\scalebox{.3}{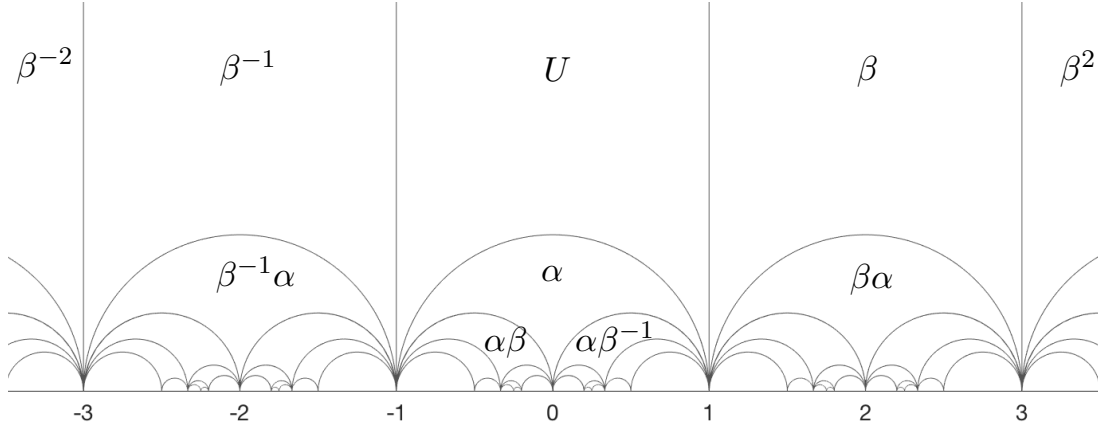}
\captionsetup{width=.9\textwidth}
\caption{Pictured is a fundamental domain $U$ for the group $\Gamma$ and images of $U$ under several elements of the group. Images of $U$ under group elements consisting of words of length $\leq2$ are labelled.}
\label{fig:fundamental_dom_group}
\end{figure}

Now let $c$ belong to the principal hyperbolic component of $\mathcal{M}$. An initial connection between $\Gamma$ and $p_c$ is as follows: the maps $p_c: \mathcal{J}(p_c) \rightarrow \mathcal{J}(p_c)$ and $\rho: \hat{\mathbb{R}} \rightarrow \hat{\mathbb{R}}$ are topologically conjugate (this is explained in Remark \ref{mark_part_rem}). This leads to the following question: is there a holomorphic map $p_c\sqcup\Gamma$ exhibiting simultaneously the dynamics of $p_c$ on $K(p_c)$, and $\Gamma$ on $\overline{\mathbb{H}}$? To state this question more precisely we will need to discuss the notion of \emph{conformal mating}, which we defer to Section \ref{conf_mating_def_sec}. Our first result (illustrated in Figure \ref{fig:scqds.pdf_tex}) is that for $c\in\mathbb{R}$, the conformal mating $p_c\sqcup\Gamma$ exists, and is in fact an object well studied in its own right.


%


\begin{thmx}\label{main_theorem} Let $c\in(-.75,.25)$. Then there exists a conformal mating $p_c\sqcup \Gamma$ as the Schwarz function of a simply connected quadrature domain. \end{thmx}

\begin{figure}
\centering
\scalebox{.15}{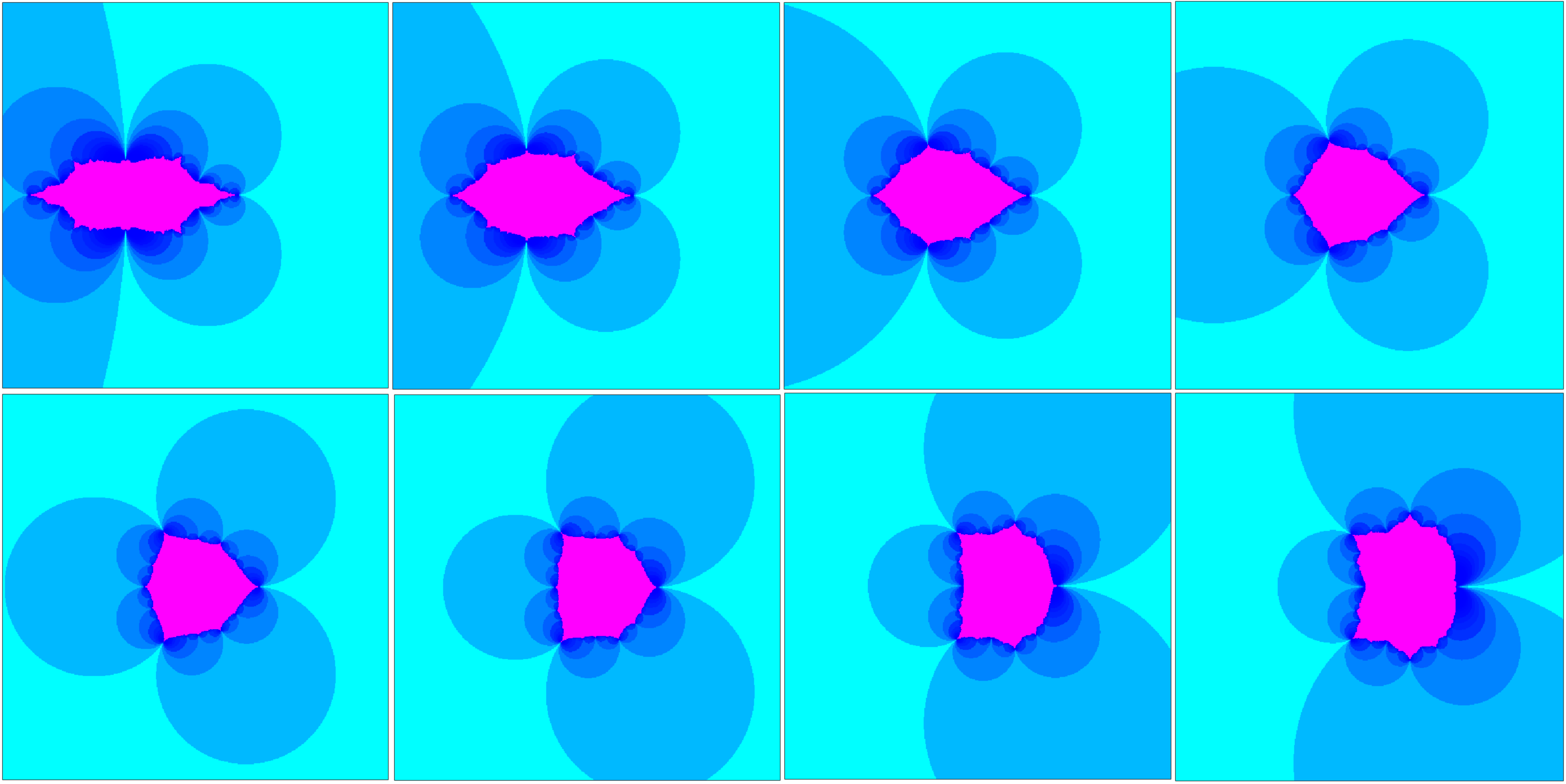}
\captionsetup{width=.9\textwidth}
\caption{Pictured is the dynamical plane of the mating $p_c\sqcup\Gamma$ as in Theorem \ref{main_theorem} with $c$ ranging in $(-.75, .25)$. The simply connected quadrature domain is the complement of the turqoise region, with Schwarz function $\sigma=p_c\sqcup\Gamma$. In pink is the \emph{filled Julia set} $K(\sigma)$ (see Definition \ref{Julia_Fatou}), where $\sigma$ is conformally conjugate to $p_c: K(p_c) \rightarrow K(p_c)$. The complement of $K(\sigma)$ is the \emph{escaping set} $I(\sigma)$ (see Definition \ref{tile_escaping}) of $\sigma$, where $\sigma$ is described by the action of the group $\Gamma$ on $\mathbb{H}$ (see Theorem \ref{group_conjugacy}). The map $\sigma$ on the common boundary of $I(\sigma)$, $K(\sigma)$ is topologically conjugate both to $p_c: \mathcal{J}(p_c) \rightarrow \mathcal{J}(p_c)$ and $\rho: \hat{\mathbb{R}} \rightarrow \hat{\mathbb{R}}$.   }
\label{fig:scqds.pdf_tex}
\end{figure}


A planar domain $\Omega$ is a \emph{quadrature domain} if there exists a meromorphic function $\sigma: \Omega \rightarrow \hat{\mathbb{C}}$ such that $\sigma$ is meromorphic in $\Omega$ and has an extension to $\overline{\Omega}$ satisfying $\sigma(z)=\overline{z}$ on $\partial\Omega$. The function $\sigma$ is then called the \emph{Schwarz function} of $\Omega$. Quadrature domains have been studied in connection to several different areas of analysis, including extremal problems for analytic functions, the Hele-Shaw flow in fluid dynamics, and potential theory. We refer to \cite{MR2129731},  \cite{MR3454377} for details and a broader overview, both of quadrature domains and the aforementioned connections. We mention two examples of quadrature domains. The first is that a simply connected domain is a quadrature domain if and only if the corresponding Riemann map is a rational function (Lemma 2.3, \cite{AS}). The second example (see the left-most curve in Figure \ref{fig:deltoid_flow.pdf_tex}) is that of a disjoint union of a cardioid with the exterior of a circle.

The study of the dynamics of the \emph{Schwarz reflection map} defined by $z\mapsto\overline{\sigma(z)}$ was initiated in \cite{MR3454377}, and several works have since studied the dynamics of $\overline{\sigma}$ for various classes of quadrature domains: see \cite{LLMM1}, \cite{LLMM2}, \cite{lee2019schwarz}, \cite{lodge2019dynamical}, \cite{LMM1}, \cite{lazebnik2020bers}. Associated with $\overline{\sigma}$ is a natural dynamical partition of $\hat{\mathbb{C}}$. In the presence of $\mathbb{R}$-symmetry, a similar partition holds for $\sigma$ (see Definitions \ref{tile_escaping}, \ref{Julia_Fatou}). Indeed, one can conclude the following from the work of \cite{LLMM2} together with a variant of Theorem \ref{group_conjugacy}:




\begin{thm}\label{main_cor} \emph{(\cite{LLMM2}) Let $c\in[-2,-.75]$ be such that $p_c$ is geometrically finite. Then there exists a conformal mating $p_c\sqcup \Gamma$ as the Schwarz function of a disjoint union of a cardioid and circle.}
\end{thm}

\noindent Taken together, Theorem \ref{main_theorem} and Theorem \ref{main_cor} yield the following description: 

\begin{cor}\label{main_cor2} \emph{Let $c\in[-2,.25)$, and assume furthermore $p_c$ is geometrically finite if $c\in[-2,-.75]$. Then there exists a conformal mating $p_c\sqcup \Gamma$ as the Schwarz function of a: \begin{enumerate} \item simply connected quadrature domain if $c\in(-.75,.25)$ \item disjoint union of a cardioid and circle if $c\in[-2,-.75]$. 
\end{enumerate} }
\end{cor}

Let us discuss cases (1) and (2) in Corollary \ref{main_cor2}. As $c\in\mathbb{R}$ passes through $-.75$, the dynamics of $p_c$ undergo the well-known phenomenon of bifurcation: the finite attracting fixed point of $p_c$ in case (1) ``becomes'' a period 2 attracting cycle as one passes to case (2). We observe here the following for the corresponding quadrature domains in (1) and (2): the quadrature domains in case (1) \emph{converge geometrically} to those in case (2). This is illustrated in Figure \ref{fig:deltoid_flow.pdf_tex}. The group structure $\Gamma$, on the other hand, is the same in cases (1) and (2).







\begin{figure}
\centering
\scalebox{.18}{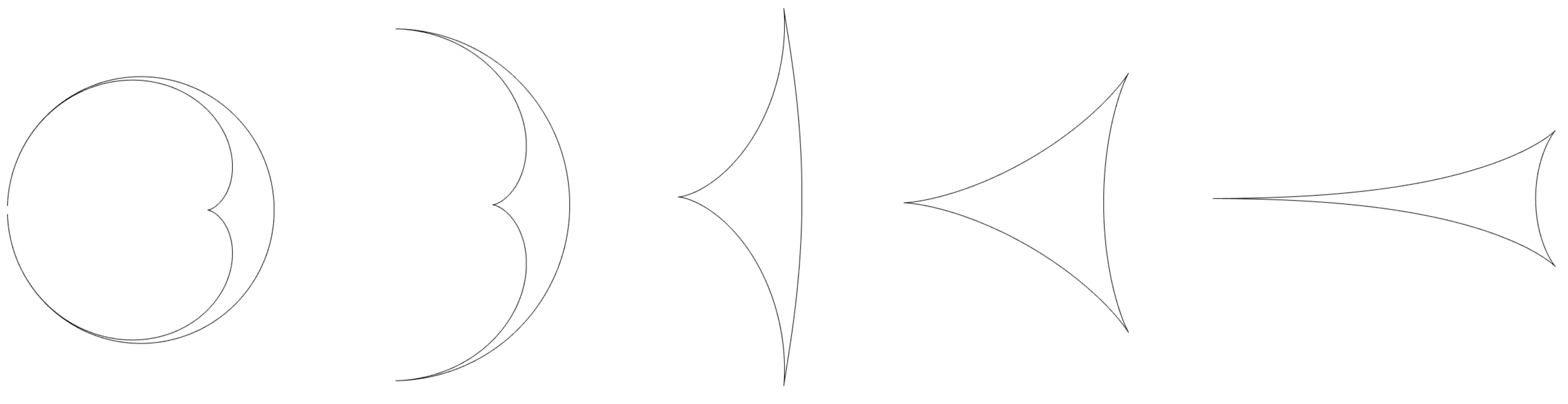}
\captionsetup{width=.9\textwidth}
\caption{Pictured are images of $\mathbb{T}$ under a 1-parameter family of maps $f$ univalent in $\mathbb{D}$. The unbounded components of $\hat{\mathbb{C}}\setminus f(\mathbb{T})$ are quadrature domains, and the bounded components of $\hat{\mathbb{C}}\setminus f(\mathbb{T})$ are usually called \emph{droplets}. This figure shows a phase transition in which a family of simply connected quadrature domains converge geometrically to a quadrature domain with two components: the cardioid and circle, approximated by the left-most (simply connected) quadrature domain.}
\label{fig:deltoid_flow.pdf_tex}
\end{figure}

In fact, we will study a broader class of holomorphic dynamical systems $\mathcal{R}$ (see Section \ref{dyn_Sch_func}) associated with a simply connected quadrature domain, which includes as a special case those systems in Theorem \ref{main_theorem}. Numerical evidence indicates a rich quadratic-like structure in the parameter space of $\mathcal{R}$: see Figure \ref{fig:parameter_comparison}.

\begin{figure}
\centering
\scalebox{.12}{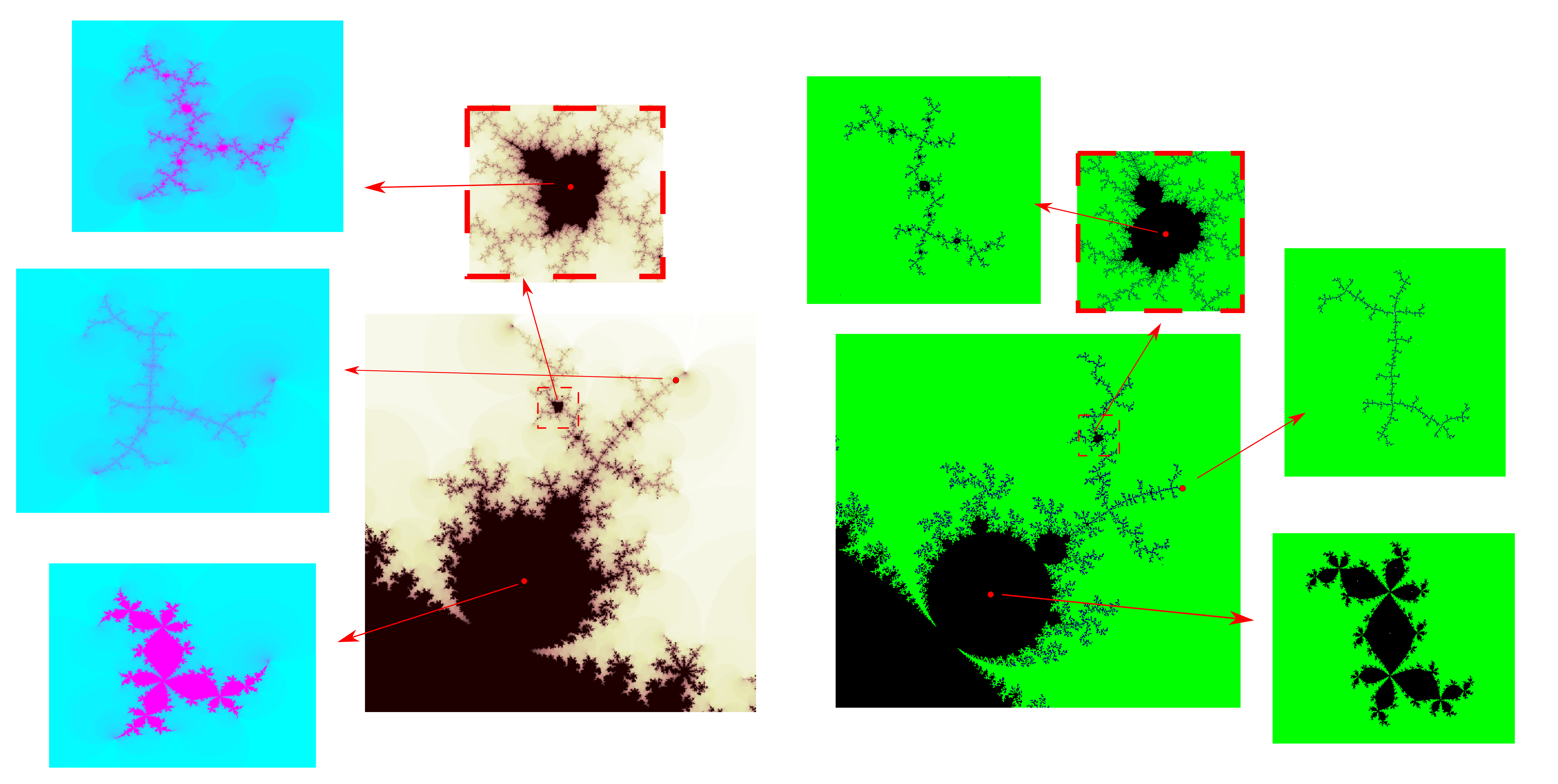}
\captionsetup{width=.9\textwidth}
\caption{Pictured on the right is a ``bulb'' of the Mandelbrot set. Also indicated are Julia sets arising from several parameters in this bulb. Pictured on the left is a parameter space of a family $\sigma_f$ as in Definition \ref{sigma_definition_larger} arising from a 1-parameter family $f$ of univalent maps. Sample filled Julia sets corresponding to several parameters are illustrated. The right-hand side figures were generated with the program \emph{Mandel} of Jung.   }
\label{fig:parameter_comparison}
\end{figure}

Further, it is possible to draw several conclusions beyond the $\mathbb{R}$-symmetric and degree $2$ setting of Theorem \ref{main_theorem}. In the following, $\Gamma_d$ denotes a Kleinian group which is a higher-degree analogue of $\Gamma$ (see Figure \ref{fig:fund_dom_next.pdf_tex}) which we will define in Section \ref{sim_unif_sec}, and $\beta(\Gamma_d)$ its Bers slice. We will sketch a proof of the following in Section \ref{sim_unif_sec}:

\begin{thmx}\label{Theorem_B} Let $d\geq2$, and $\beta\in\beta(\Gamma_d)$. Suppose $p\in\textrm{Pol}_d$ is in the main hyperbolic component in the parameter space of $\textrm{Pol}_d$. Then there exists a conformal mating $p_c\sqcup \beta$ of $p_c$ with $\beta$. 
\end{thmx}

\noindent The main idea of the proof is that of simultaneous uniformization due to Bers (see \cite{10.1112/blms/4.3.257}): after verifying the statement of Theorem \ref{Theorem_B} for a ``base case'' of $p(z)=z^d$ and $\Gamma_d$ (where the mating is a Schwarz function $\sigma$), the more general setting of Theorem \ref{Theorem_B} will follow from applying quasiconformal surgery simultaneously in the group and polynomial parts of the dynamical plane of $\sigma$.  Indeed, a similar strategy will be applied in the proof of Theorem \ref{main_theorem}, although in the setting of Theorem \ref{main_theorem} the group structure is rigid and so we can only vary the polynomial dynamics. We emphasize that although the conformal mating  $p_c\sqcup \beta$ is shown to exist in the proof of Theorem \ref{Theorem_B}, a more concrete description of the mating is not given (unlike the Schwarz function description given in Theorem \ref{main_theorem}).

We also remark that a notion of conformal matings between Kleinian groups and polynomials as holomorphic correspondences was introduced in \cite{Bullett1994} (see also \cite{BL20}). The definition of conformal mating used in the present work follows more closely that of \cite{LLMM1}. In particular, the map $p_c\sqcup\Gamma$ as in Theorems \ref{main_theorem}, \ref{Theorem_B} is a single-valued holomorphic function, rather than a multi-valued correspondence as in \cite{Bullett1994}. There is, however, a natural extension of the Schwarz function of a simply connected domain as a correspondence on $\hat{\mathbb{C}}$. Thus a natural question is whether this correspondence is a mating of $p_c$, $\Gamma$ in the sense of \cite{Bullett1994}, some branch of which is the Schwarz function of Theorem \ref{main_theorem}. If the answer is affirmative, this would provide a satisfying link between the two points of view. 

We now briefly outline the paper. In Sections \ref{dyn_Sch_func} and \ref{hyp_dyn_section} we introduce the class of dynamical systems $\mathcal{R}$ and establish some basic properties. In Section \ref{mod_group_struct} we study the group $\Gamma$, leading to a definition of conformal mating in Section \ref{conf_mating_def_sec}. In Section \ref{topology} we study the topology of the filled Julia set of a base element in $\mathcal{R}$. Section \ref{qc_surgery_sec} is devoted to quasiconformal deformations of this base element of $\mathcal{R}$ and the proof of Theorem \ref{main_theorem}. Section \ref{sim_unif_sec} is devoted to Theorem \ref{Theorem_B}, and lastly in Section \ref{remarks_section} we remark on some questions naturally arising from the present work.



\section{Dynamics of Schwarz Functions}\label{dyn_Sch_func}


\begin{notation} We use the notation: 

\begin{align} \nonumber \textrm{Pol}_d:= \left\{ p(z):=a_dz^d+...+a_1z+a_0 : a_d, ..., a_0 \in \mathbb{C}\textrm{, and } a_d\not=0 \right\}, \hspace{2mm} \textrm{Pol} := \bigcup_d \textrm{Pol}_d\textrm{, and} \nonumber\\ \textrm{Rat}_{d} := \left\{ \frac{p(z)}{q(z)} : p, q\in\textrm{Pol} \textrm{ and } \textrm{max}(\textrm{deg}(p), \textrm{deg}(q)) = d \right\}. \phantom{asdfasasddfs}  \nonumber \end{align}

\end{notation}

\begin{notation} Given three points $z_1$, $z_2$, $z_3 \in \hat{\mathbb{C}}$, we denote by $D_{z_1, z_2, z_3}$ the open Euclidean disc such that $z_1$, $z_2$, $z_3 \in \partial D_{z_1, z_2, z_3}$ and $(z_1, z_2, z_3)$ is oriented positively with respect to $D_{z_1, z_2, z_3}$.  Given a Euclidean disc $D$, we denote $D^*:=\hat{\mathbb{C}}\setminus\overline{D}$.
\end{notation}


\begin{definition}\label{main_family_larger} We define a collection $\mathcal{R}$ whose elements are pairs $(f, (c_1, c_2, c_3))$, where $f\in\textrm{Rat}_{3}$ is such that $f'$ has only simple zeros, and $(c_1, c_2, c_3)$ are critical points of $f$ such that $f$ is univalent (injective) on $\overline{D_{c_1, c_2, c_3}}$.
\end{definition}

\begin{rem} We will frequently omit the triple $(c_1, c_2, c_3)$ from our notation and simply write $f\in\mathcal{R}$. Given a choice of $c_1$, $c_2$, $c_3$ as in  Definition \ref{main_family_larger}, we denote by $c_f$ the remaining critical point of $f$.
\end{rem}

\begin{definition}\label{sigma_definition_larger} Let $(f, (c_1, c_2, c_3))\in\mathcal{R}$, and $D=D_{c_1, c_2, c_3}$. Let $M$ denote the M\"obius transformation determined by: \begin{enumerate} \item $M(c_1)=c_1$, \item $M(c_2)=c_3$, \item $M(c_3)=c_2$. \end{enumerate} We define $\sigma_f: f(D) \rightarrow \hat{\mathbb{C}}$ by the following diagram:

\[
  \begin{tikzcd}
    D \arrow{r}{M} & D^* \arrow{d}{f} \\
     f(D ) \arrow{r}{\sigma_f} \arrow{u}{f^{-1}}  & \widehat{\mathbb{C}}
  \end{tikzcd}
\]
\end{definition} 

\begin{rem} The definition of $\sigma_f$ depends not only on $f\in\mathcal{R}$, but also a choice (and ordering) of the critical points $c_1$, $c_2$, $c_3$.
\end{rem}

\begin{rem} When $f$ is $\mathbb{R}$-symmetric, and $c_1\in\mathbb{R}$ with $c_2=\overline{c_3}$, the map $\sigma_f$ is the Schwarz function of the quadrature domain $f(D)$. Thus the family $\mathcal{R}$ contains the Schwarz functions of simply connected $\mathbb{R}$-symmetric quadrature domains.
\end{rem}


\begin{definition}\label{tile_escaping} Let $f\in\mathcal{R}$. We define the \emph{fundamental tile} of $f$ by \[ T_f:=\hat{\mathbb{C}}\setminus \left( f(D) \cup \left\{f(c_1), f(c_2), f(c_3) \right\} \right). \]

\noindent The \emph{escaping set} of $\sigma_f$ is defined by: \[ I(\sigma_f):= \{ z\in\hat{\mathbb{C}} : \sigma_f^n(z)\in T_f \textrm{ for some } n\geq0 \}. \]
\end{definition}

\begin{prop}\label{first_covering_property} Let $f\in\mathcal{R}$, and denote $U:=f(D) \setminus \overline{\sigma_f^{-1}(T_f)}$. Then \begin{equation}\nonumber \sigma_f: \sigma_f^{-1}(\emph{int}\hspace{.5mm}T_f) \rightarrow \emph{int}\hspace{.5mm}T_f \end{equation} is a degree 3 covering map, and \begin{equation}\nonumber \sigma_f: U \rightarrow f(D)  \end{equation} is a degree 2 branched covering map, branched only at the point $f\circ M(c_f)$. 
\end{prop}




\begin{proof} Note that $f:\hat{\mathbb{C}}\rightarrow\hat{\mathbb{C}}$ is $3:1$, and $f:D\rightarrow f(D)$ is $1:1$. Thus the degree statements follow from the diagram of Definition \ref{sigma_definition_larger}. Note that $c_f\not\in\partial D_{c_1, c_2, c_3}$, since otherwise, $f$ would be univalent in $D^*$ which is impossible. Thus $c_f\in D^*$, and so $\sigma_f$ is branched at $f\circ M(c_f)$. It remains to show that $f\circ M(c_f) \in U$, for which it suffices to show that $f(c_f) \in f(D)$, and this follows from covering properties of $f$.

\end{proof}


\begin{definition}\label{Julia_Fatou} Let $f\in\mathcal{R}$. We define the \emph{filled Julia set} of $\sigma_f$ by: \[ K(\sigma_f):=\hat{\mathbb{C}}\setminus I(\sigma_f). \] The \emph{Julia set} and \emph{Fatou set} of $\sigma_f$ are defined by \[\mathcal{J}(\sigma_f):=\partial K(\sigma_f)\textrm{, } \hspace{3mm} \mathcal{F}(\sigma_f):=\textrm{int}\hspace{.5mm}K(\sigma_f) \] respectively.

\end{definition}


\begin{rem} We will sometimes omit the dependence on $f$ from our notation when $f$ is clear from the context, and write, for instance, $T$ in place of $T_f$, or $\sigma$ in place of $\sigma_f$.
\end{rem}

\begin{prop}\label{openess_of_escaping_set} Let $f\in\mathcal{R}$. Then $I(\sigma_f)$ is open. \end{prop}

\begin{proof} We will show that $I(\sigma)\subset \textrm{int}\hspace{.5mm}I(\sigma)$. Let $z\in I(\sigma)$. Suppose first that $z\in T$. If $z\not\in f(\partial D)$, it is clear that $z\in\textrm{int}\hspace{.5mm}T \subset \textrm{int}\hspace{.5mm}I(\sigma)$, so we assume that $z\in f(\partial D)$. Since $f(\partial D)\setminus\{f(c_1), f(c_2), f(c_3) \}$  is preserved set-wise by $\sigma$, we have $\sigma(z)\in f(\partial D)\setminus\{f(c_1), f(c_2), f(c_3) \}$. Thus for a small neighborhood $U$ of $z$, we have that \[\sigma(U\cap f(D))\subset\hat{\mathbb{C}}\setminus \overline{f(D)}.\] Thus  $U\cap f(D) \subset I(\sigma)$, and so $z\in \textrm{int}\hspace{.5mm}I(\sigma)$. Similar reasoning shows that if $z\in \sigma^{-n}(I(\sigma))$ for some $n>0$, then $z\in I(\sigma)$.
\end{proof}

\begin{rem} Proposition \ref{openess_of_escaping_set} indicates the importance of the requirement that the M\"obius transformation $M$ of Definition \ref{sigma_definition_larger} preserves (set-wise) the critical points $c_1$, $c_2$, $c_3$. Without this requirement, one would have, for instance, isolated points in the Julia set (generically).

\end{rem}

\begin{prop}\label{compactness_of_K} Let $f\in\mathcal{R}$. Then $K(\sigma_f)$ is compact, full, and \begin{equation}\label{filled_julia_decomposition} K(\sigma_f) = \mathcal{J}(\sigma_f) \sqcup \mathcal{F}(\sigma_f). \end{equation}
\end{prop}

\begin{proof} That $K(\sigma)$ is closed follows from Proposition \ref{openess_of_escaping_set}. Thus $K(\sigma)$ is compact with respect to the spherical metric on $\hat{\mathbb{C}}$. Relation (\ref{filled_julia_decomposition}) follows. Now observe once more that that $f(\partial D)\setminus\{f(c_1), f(c_2), f(c_3) \}$ is preserved set-wise by $\sigma$. Thus each of the three components of $\sigma^{-1}(T)$ have non-empty intersection with $T$, and so $\sigma^{-1}(T)\cup T$ is connected. Similar reasoning shows inductively that \[ \bigcup_{n=0}^k \sigma^{-n}(T) \] is connected for each $k\geq0$. Thus connectivity of $I(\sigma)$ follows, and so $K(\sigma)$ is full.


\end{proof}

\begin{prop}\label{invariance} Let $f\in\mathcal{R}$. Then $\mathcal{J}(\sigma_f)$ and $\mathcal{F}(\sigma_f)$ are both totally invariant under $\sigma_f$.
\end{prop}

\begin{proof} That $\mathcal{F}(\sigma)$ is totally invariant follows from the observation that $\sigma$ is an open mapping. As \[\sigma(\mathcal{J}(\sigma))\subset K(\sigma),\] total invariance of $\mathcal{J}(\sigma)$ now follows from total invariance of $\mathcal{F}(\sigma)$.
\end{proof}

\begin{cor}\label{normality} Let $f\in\mathcal{R}$. Then the family $(\sigma_f^{n})_{n=1}^\infty$ is normal in $\mathcal{F}(\sigma_f)$. 
\end{cor}

\begin{proof} Given $z\in\mathcal{F}(\sigma)$, all iterates of $\sigma^n$ are defined in a sufficiently small neighborhood $U$ of $z$ by definition of $\mathcal{F}(\sigma)$. As $\sigma^n(U)\subset f(D)$ for all $n$, and $\hat{\mathbb{C}}\setminus \overline{f(D)}$ is open, the result follows from elementary criteria for normality. 
\end{proof}

\begin{prop}\label{conn_locus} Let $f\in\mathcal{R}$. Then $K(\sigma_f)$ is connected if and only if $c_f\not\in I(\sigma_f)$. 
\end{prop}

\begin{proof} Let $U:=f(D) \setminus \overline{\sigma_f^{-1}(T_f)}$ as in Proposition \ref{first_covering_property}. The map $\sigma_f: U \rightarrow f(D)$ is a degree 2 branched covering map by Proposition \ref{first_covering_property}. Letting $\chi$ denote the Euler Characteristic, we have then that, by the Riemann-Hurwitz formula, \[2\cdot\chi(f(D))-\chi(U) = \#\{\textrm{critical points of } \sigma_f \textrm{ in } U \}. \] As $\sigma_f$ has 1 critical point in $U$ by Proposition \ref{first_covering_property}, we see that $\chi(U)=1$, in other words $U$ is connected. More generally, we see that $2\cdot\chi(\sigma^{-n}(U))-\chi(\sigma^{-n-1}(U))$ is equal to the number of critical points of $\sigma_f$ in $\sigma^{-n-1}(U)$. Thus as $\chi(U)=1$, the conclusion follows.


\end{proof}

\begin{rem} Proposition \ref{conn_locus} defines a connectedness locus for the family $\mathcal{R}$. Figure \ref{fig:parameter_comparison} shows a zoom of the locus for a 1-dimensional subfamily of $\mathcal{R}$, in comparison with the connectedness locus for the quadratic family.
\end{rem}


\section{Hyperbolic Dynamics}\label{hyp_dyn_section}



 



\begin{definition} Let $f\in\mathcal{R}$, and suppose that $p\in f(D)$ is such that $\sigma_f(p)=p$. We will call $\lambda:=\sigma_f'(p)$ the \emph{multiplier} of $p$. If $|\lambda|<1$, we call $p$ \emph{attracting}, and define the \emph{basin of attraction} $\mathcal{A}=\mathcal{A}(p)$ to be the set of $x\in K(\sigma_f)$ such that $\lim_{n\rightarrow\infty}\sigma_f^n(x)$ exists and is equal to $p$. The \emph{immediate basin of attraction} $\mathcal{A}_0$ is defined to be the connected component of $\mathcal{A}$ containing $p$.
\end{definition}


\begin{thm}\label{Koenig_coord} {\bf (K\oe nig's Linearization)} Let $f\in\mathcal{R}$ be such that $\sigma_f$ has a fixpoint at $p$ of multiplier $\lambda$ where $0<|\lambda|<1$. Then there is a holomorphic map $\phi: \mathcal{A}\rightarrow\mathbb{C}$ with $\phi(p)=0$ such that the diagram \[
  \begin{tikzcd}
    \mathcal{A} \arrow{r}{\sigma_f} \arrow[swap]{d}{\phi}  & \mathcal{A} \arrow{d}{\phi} \\
     \mathbb{C} \arrow{r}{\lambda\cdot}  & \mathbb{C}
  \end{tikzcd}
\] commutes. Moreover, $\phi$ is conformal in a neighborhood of $0$. Let $\phi^{-1}$ denote the inverse of $\phi$ defined near $0$. There is a maximal open disc $\mathbb{D}_r$ centered at the origin to which $\phi^{-1}$ extends conformally. The map $\phi^{-1}$ extends homeomorphically to $\partial\mathbb{D}_r$, and $\phi^{-1}(\partial\mathbb{D}_r)$ contains a critical point of $\sigma_f$.
\end{thm}

\begin{proof} We need only observe that for such a $p$, we have that $p\in\mathcal{F}(\sigma_f)$ and hence $\mathcal{A}\subset\mathcal{F}(\sigma_f)$. Thus denoting the Riemann surface $S:=\mathcal{F}(\sigma_f)$, we have $\sigma_f: S \rightarrow S$ by Proposition \ref{invariance}. The conclusions listed in the Theorem now follow directly from applying Theorem 8.2, Corollary 8.4, and Lemma 8.5 of \cite{MR2193309} to the holomorphic dynamical system $\sigma_f: S \rightarrow S$.
\end{proof}

\begin{thm}\label{bottcher_coord}{\bf (B\"ottcher)} Let $f\in\mathcal{R}$ be such that $\sigma_f(p)=p$, $\sigma_f'(p)=0$, and let $\mathcal{A}:=\mathcal{A}(p)$. Then there is a conformal map $\phi: \mathcal{A} \rightarrow \mathbb{D}$ such that the diagram \[
  \begin{tikzcd}
    \mathcal{A} \arrow{r}{\sigma_f} \arrow[swap]{d}{\phi}  & \mathcal{A} \arrow{d}{\phi} \\
     \mathbb{C} \arrow{r}{z\mapsto z^2}  & \mathbb{C}
  \end{tikzcd}
\] commutes. 
\end{thm}

\begin{proof} As in the proof of Theorem \ref{Koenig_coord}, we consider the holomorphic dynamical system $\sigma_f: S \rightarrow S$ with $S:=\mathcal{F}(\sigma_f)$. The existence of a holomorphic map $\phi$ satisfying the conclusions of the Theorem follow directly from Theorem 9.1 of \cite{MR2193309}. By  Proposition \ref{first_covering_property}, $\sigma_f$ has no other critical points in $S$, and so by Theorem 9.3 of \cite{MR2193309}, the map $\phi$ is conformal in $\mathcal{A}$.
\end{proof}

\begin{rem} We will refer to $\phi$ as in the conclusion of Theorem \ref{Koenig_coord} as the K\oe nig coordinate for $\sigma_f$, and $\phi$ as in the conclusion of Theorem \ref{bottcher_coord} as the B\"ottcher coordinate for $\sigma_f$.
\end{rem} 




\begin{thm}\label{Fatou_set_conjugacy} Let $f\in\mathcal{R}$, $p\in\textrm{Pol}_2$ be such that $\sigma_f$ and $p$ each have an attracting fixed point of the same multiplier, with basins of attraction $\mathcal{A}_\sigma$, $\mathcal{A}_p$, respectively. Then there is a conformal map $\psi: \mathcal{A}_{\sigma} \rightarrow \mathcal{A}_p$ such that $ \psi\circ\sigma_f\circ\psi^{-1}\equiv p \textrm{ on } \mathcal{A}_p$. 
\end{thm}

\begin{proof}

We first assume that $\lambda\not=0$. Let $\phi_\sigma$, $\phi_p$ denote the K\oe nig coordinates for $\sigma$, $p$, respectively. Let $\mathbb{D}_r$, as in Proposition \ref{Koenig_coord}, denote the maximal disc centered at $0$ to which $\phi_\sigma^{-1}$ extends conformally. By multiplying $\phi_\sigma$ by a real number, we may assume that $\mathbb{D}_r$ is also the maximal disc centered at $0$ to which $\phi_{p}^{-1}$ extends conformally. Thus, by Lemma 8.5 of \cite{MR2193309}, $\phi_p^{-1}(\partial \mathbb{D}_r)$ contains a critical point $c_p$ of $p$, and by Theorem \ref{Koenig_coord}, $\phi_\sigma^{-1}(\partial \mathbb{D}_r)$ contains the critical point $c_\sigma$ of $\sigma$. By multiplying $\phi_\sigma$ by a unimodular constant, we may assume that \begin{equation}\label{crit_pts_correspond}\phi_p^{-1}\circ\phi_\sigma(c_\sigma)=c_p.\end{equation} 

Now, as $\phi_\sigma$, $\phi_p$ are K\oe nig coordinates, we have \begin{equation}\label{first_conj} (\phi_p^{-1} \circ \phi_\sigma) \circ \sigma \circ (\phi_p^{-1} \circ \phi_\sigma)^{-1} \equiv p \textrm{ on } \phi_p^{-1}\left(\overline{\mathbb{D}_r}\right). \end{equation} Thus by (\ref{crit_pts_correspond}) and (\ref{first_conj}), we have that \begin{equation}\label{crit_val_correspond} \phi_p^{-1} \circ \phi_\sigma(\sigma(c_\sigma)) =   p \circ \phi_p^{-1} \circ \phi_\sigma(c_\sigma) =  p(c_p). \end{equation} Moreover, as $\sigma$ has only one critical point by Proposition \ref{first_covering_property}, \begin{equation}\label{coverings}  \sigma: \mathcal{A}\setminus\{c_\sigma\} \rightarrow \mathcal{A}\setminus\{\sigma(c_\sigma)\}  \textrm{ and } p: \mathcal{A}\setminus\{c_p\} \rightarrow \mathcal{A}\setminus\{p(c_p)\} \end{equation} are both (unbranched) covering maps. Thus, using (\ref{crit_pts_correspond}) and (\ref{crit_val_correspond}), we have by iterative lifting under the coverings (\ref{coverings}), that\[  \phi_p^{-1}\circ\phi_\sigma: \phi_\sigma^{-1}(\mathbb{D}_r)\setminus\{ \sigma(c_\sigma) \} \rightarrow  \phi_p^{-1}(\mathbb{D}_r)\setminus\{ p(c_p) \}    \] extends to a conformal map \[ \phi_p^{-1}\circ\phi_\sigma:  \mathcal{A}\setminus\{c_\sigma, \sigma(c_\sigma)\} \rightarrow \mathcal{A}\setminus\{c_p,p(c_p)\} \textrm{ satisfying }  (\phi_p^{-1} \circ \phi_\sigma) \circ \sigma \circ (\phi_p^{-1} \circ \phi_\sigma)^{-1} \equiv p\textrm{ on } \mathcal{A}\setminus\{c_p,p(c_p)\}.\]  The above relation was already shown to hold at the points $c_p$, $p(c_p)$ in (\ref{crit_pts_correspond}), (\ref{crit_val_correspond}), so that we have proven the existence of a conformal conjugacy $\phi_p^{-1}\circ\phi_\sigma: \mathcal{A}_\sigma\rightarrow \mathcal{A}_p$ between $\sigma|_{\mathcal{A}}$, $p|_{\mathcal{A}}$. 



Lastly, we remark that in the case $\lambda=0$, one readily checks that $\psi:=\phi_p^{-1}\circ\phi_\sigma$ is the desired conjugacy, where $\phi_\sigma$, $\phi_p$ now denote the B\"ottcher coordinates for $\sigma$, $\phi$, respectively. 





\end{proof}



\section{Modular Group Structure}\label{mod_group_struct}

\begin{notation} We will use the notation $\alpha$, $\beta$ to denote the maps:
\begin{equation} \alpha(z):= \frac{-1}{z}\textrm{, }\hspace{3mm}\beta(z):=z+2. \end{equation} We denote by $\Gamma$ the subgroup of $\textrm{PSL}(2,\mathbb{Z})$ generated by $\alpha, \beta \in \textrm{PSL}(2,\mathbb{Z})$. 
\end{notation}


\begin{notation} We denote by: \begin{equation} U:= \bigg\{ z \in \mathbb{H} : -1 < \textrm{Re}(z) < 1 \bigg\} \cap \left\{ z \in \mathbb{H} : \left|z\right| > 1 \right\}\textrm{, and} \end{equation}
\begin{equation} D:= U \cup \{ z \in \mathbb{H} : \textrm{Re}(z)=-1 \} \cup \{ z \in \mathbb{H} : |z| = 1, \hspace{1mm} \textrm{Re}(z)\leq0 \}  \end{equation} 

\end{notation}

\begin{prop}\label{fundamental_domain} The set $D$ is a fundamental domain for $\Gamma$. 
\end{prop}


\begin{proof} First we prove that for any $z\in\mathbb{H}$, there exists $g\in\Gamma$ with $gz\in D$. Note the identity: \begin{equation}\label{max_imag} \textrm{Im}\left( \frac{az+b}{cz+d} \right) = \frac{\textrm{Im}(z)}{|cz+d|^2} \textrm{ for } \begin{pmatrix} a & b \\ c & d \end{pmatrix} \in \textrm{PSL}(2,\mathbb{Z}). \end{equation} Let $z:=x+iy$. Then $|cz+d|^2=(cx+d)^2+(cy)^2$. Thus, as $h$ ranges over all elements of $\Gamma$ and $z$ is fixed, $|cz+d|^2$ stays bounded away from $0$. Hence, by (\ref{max_imag}), there exists $g_0\in\Gamma$ such that $\textrm{Im}(hz)\leq\textrm{Im}(g_0z)$ for all $h\in\Gamma$. In particular, we have $\textrm{Im}(\alpha g_0z)\leq\textrm{Im}(g_0z)$. Thus \begin{equation} \textrm{Im}\left( \alpha g_0z \right) = \frac{\textrm{Im}(g_0z)}{|g_0z|^2} \leq \textrm{Im}(g_0z), \end{equation} and so $|g_0z|^2\geq1$. It follows then that for some $n\in\mathbb{Z}$ we have $\beta^ng_0z\in D$ as needed.


Secondly, we prove that if $z\in D$, then $gz\not\in D$ for all non-identity elements $g\in\Gamma$. We will appeal to the standard result that \[ \tilde{D}:=D \cap \left\{ z \in \mathbb{H} : -1/2 \leq \textrm{Re}(z) < 1/2 \right\} \] is a fundamental domain for $\textrm{PSL}(2,\mathbb{Z})$. The proof breaks up into several cases: namely we wish to show that if $z$ belongs to any of the regions $\tilde{D}$, $T^{-1}(\tilde{D})\cap D$, $T(\tilde{D})\cap D$, $T^{-1}S(\tilde{D}) \cap D$ or $TS(\tilde{D})\cap D$ (whose union constitutes $D$: see Figure \ref{fig:fund_dom_mod_group}) then $gz\not\in D$.  We focus on the case that $z\in T^{-1}\tilde{D} \cap D$ as the other cases are similar.

We will use the notation $T(z):=z+1$, and $S(z):=\alpha(z)$. Let $g\in\Gamma$ be a non-identity element. Since $T^{-1}\tilde{D}$ is also a fundamental domain for $\textrm{PSL}(2,\mathbb{Z})$, it can not be the case that $gz\in T^{-1}\tilde{D}$. Similarly, $gz\not\in\tilde{D}$ since otherwise we would have $g=T\not\in\Gamma$. If $gz\in T\tilde{D} \cap D$, then $g=T^2$, but $T^2(T^{-1}\tilde{D} \cap D)\cap D=\emptyset$. We also have $T^{-1}ST(T^{-1}\tilde{D} \cap D)\cap D=\emptyset$, so that by the same reasoning, $gz\not\in T^{-1}S(\tilde{D})$. Lastly, if $gz\in TS(\tilde{D})$, then $g=TST$. But then $Sg=T^{-1}S^{-1}\in\Gamma$ (where we have used the identity $(ST)^3=I$), and so $SgS=T^{-1} \in \Gamma$, and this is a contradiction as $\Gamma \subsetneq \textrm{PSL}(2,\mathbb{Z})$. Thus we have proven that if $z\in T^{-1}\tilde{D} \cap D$ and $g\in\Gamma$, then $gz\not\in D$. In the other cases: $z\in \tilde{D}$, $z\in T(\tilde{D})\cap D$, $z\in T^{-1}S(\tilde{D}) \cap D$ or $z\in TS(\tilde{D})\cap D$, one similarly verifies that $gz\not\in D$.


\begin{figure}
\centering
\scalebox{.3}{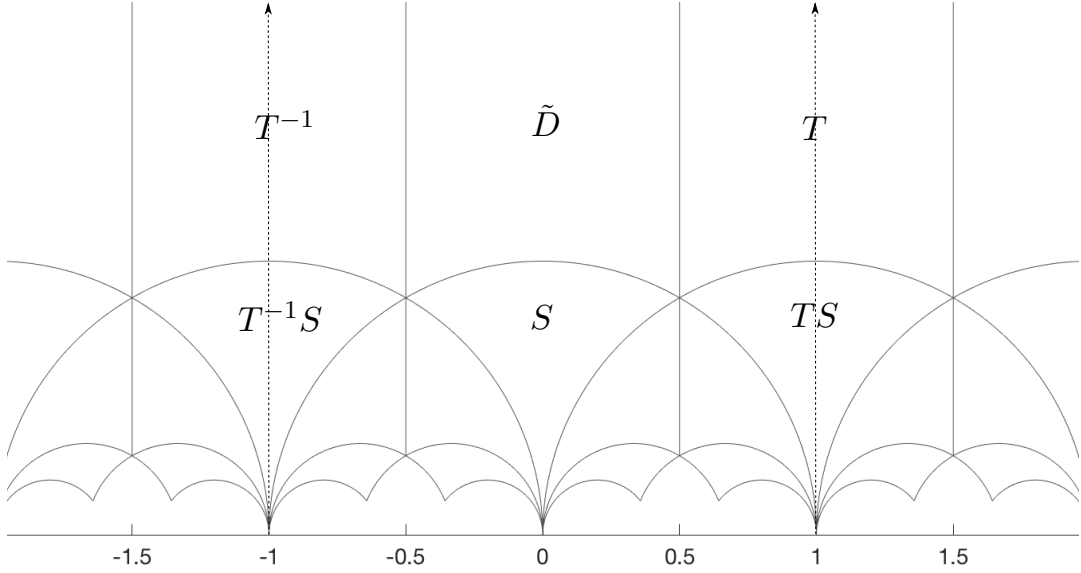}
\captionsetup{width=.9\textwidth}
\caption{An overlay of the fundamental domain $\tilde{D}$ for the group $\textrm{PSL}(2,\mathbb{Z})$ with the fundamental domain for $\Gamma$ outlined by the dotted lines and $(|z|=1)\cap\mathbb{H}$ (see also Figure \ref{fig:fundamental_dom_group}). }
\label{fig:fund_dom_mod_group}
\end{figure}


\end{proof}

\begin{definition}\label{rho_definition} Define the map $\rho: \mathbb{H}\setminus\overline{D}\rightarrow\mathbb{H}$ by 
\begin{align}
 \rho(z) := \left\{
        \begin{array}{ll}
        \alpha(z) & \quad \textrm{ if } |z|<1, \\            
        \beta^{-1}(z) & \quad \textrm{ if } \textrm{Re}(z)>1, \\
        \beta(z) & \quad \textrm{ if } \textrm{Re}(z)<-1. \\        
        \end{array}
    \right.
     \end{align}
\end{definition}

\begin{prop} Let $g\in\Gamma$, and let $n$ be the length of the word defining $g$. Then $\rho^n\circ g(D)=D$.
\end{prop}

\begin{proof} We induct on the length of the word defining $g\in\Gamma$. In the base case (if $g=\alpha$, $\beta$, or $\beta^{-1}$) then the statement is clear (with $n=1$) from the definition of $\rho$. In the inductive case, we let $g=g_1...g_n$ with each $g_i\in\{\alpha, \beta, \beta^{-1}\}$. Suppose $g_1=\beta$. Then $g_2....g_n(D)\not\in \{ z \in \mathbb{H} : \textrm{Re}(z) < -1\}$, and so $g(D)\in\{z \in \mathbb{H} : \textrm{Re}(z) > 1\}$. Then $\rho\circ g(D):=\beta^{-1}\circ g(D)=g_2....g_n(D)$, whence the inductive hypothesis applies to show that $\rho^n\circ g(D)=D$. The proof in the case that $g_1=\beta^{-1}$ or $g_1=\alpha$ is similar. 
\end{proof}



\begin{definition}\label{regular} Let $f=(f, (c_1, c_2, c_3))\in\mathcal{R}$. We say $f$ is \emph{regular} if \begin{enumerate} \item $c_1=1$, \item $D_{c_1, c_2, c_3}=\mathbb{D}$, \item $c_f\not\in I(\sigma_f)$, and \item $f(\overline{z})=\overline{f(z)}$. \end{enumerate} 
\end{definition}

\begin{thm}\label{group_conjugacy} Let $f\in\mathcal{R}$ be regular. Then there exists a conformal map $\phi: \mathbb{H} \rightarrow I_f$ such that $\phi(\overline{D})=\overline{T_f}$ and \begin{equation} \rho(z) = \phi^{-1}\circ\sigma_f\circ\phi(z) \textrm{ for } z\in \mathbb{H}\setminus\overline{D}. \end{equation}
\end{thm}

\begin{figure}
\centering
\scalebox{.2}{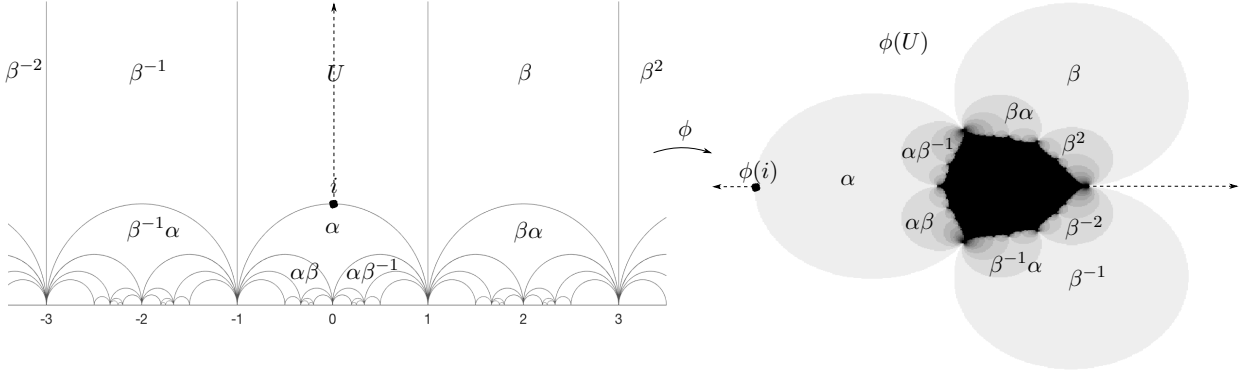}
\captionsetup{width=.9\textwidth}
\caption{Illustrated is the definition of the map $\phi$ in Theorem \ref{group_conjugacy}. Words in $\alpha$, $\beta$ on the right indicate images under $\phi$ of the corresponding tiles in $\mathbb{H}$.}
\label{fig:definition_of_phi.png}
\end{figure}

\begin{proof} As $f$ is regular, $f$ has two non-real critical points on $\mathbb{T}$ which are symmetric with respect to $\mathbb{R}$. We will denote these critical points $c$, $\overline{c}$, where $\textrm{Im}(c)>0$. We first define, by application of the Riemann mapping theorem\footnote{We note that a definition of $\phi$ may be computed explicitly in terms of $\sinh$ and $f^{-1}$. }, a conformal map \begin{equation} \phi: U \cap \{z \in \mathbb{H} : \textrm{Re}(z)>0\} \rightarrow \mathbb{H}\setminus \overline{f(\mathbb{\mathbb{D}})}\end{equation} with \begin{equation}\phi(\infty)=f(1), \hspace{2mm} \phi(1)=f(c) \textrm{ and } \phi(i)=f(-1). \end{equation} 


\noindent Next, we extend the definition of $\phi$ to $U$ as follows. Let $\iota:\mathbb{C}\rightarrow\mathbb{C}$ denote reflection in the imaginary axis. Then using the Schwarz reflection principle (in the imaginary axis in the $z$-plane, and the real-axis in the $w$-plane), we may extend the definition of $\phi$ to a conformal map \begin{equation}\label{symmetry_on_phi} \phi: U \rightarrow \mathbb{C}\setminus \overline{f(\mathbb{\mathbb{D}})} \textrm{ satisfying } \phi \circ \iota(z) = \overline{\phi(z)} \textrm{ for } z\in U. \end{equation} It follows then that $\phi(\overline{D})=\overline{T_f}$, as needed.

The maps \[\rho: \rho^{-1}(U) \rightarrow U \textrm{ and } \sigma_f: \sigma_f^{-1}(\textrm{int}\hspace{.5mm}T_f) \rightarrow \textrm{int}\hspace{.5mm}T_f \] are both degree 3 covers (see Proposition \ref{first_covering_property}), so that we can lift $\phi$ by the covers $\rho$, $\sigma_f$. Let $\hat\phi$ denote the unique such lift so that the extension of $\hat\phi$ to $\overline{\rho^{-1}(U)}$ agrees with $\phi$ at the points $0$, $1$, $\infty$. Extend the definition of $\phi$ to $\rho^{-1}(U)$ by $\phi(z):=\hat\phi(z)$ for $z\in\rho^{-1}(U)$. In fact, we claim that $\phi$ and $\hat\phi$ agree on the lines $x=1$ and $x=-1$: this follows directly from relation (\ref{symmetry_on_phi}) and the fact that $\hat\phi$ is a lift of $\phi$. Similar considerations show that $\phi$, $\hat\phi$ agree on the semi-circle $\mathbb{H}\cap\{z \in\mathbb{C} : |z| = 1\}$. Thus by removability of analytic arcs for conformal mappings, we have that: \[ \phi: D \cup \rho^{-1}(D) \rightarrow T_f \cup \sigma_f^{-1}(\textrm{int}\hspace{.5mm}T_f) \textrm{ is conformal.}\]

Recall that $f$ is assumed to regular (Definition \ref{regular}), and so in particular $c_f\not\in I(\sigma_f)$. Thus \[\rho: \rho^{-2}(U) \rightarrow \rho^{-1}(U) \textrm{ and } \sigma_f: \sigma_f^{-2}(\textrm{int}\hspace{.5mm}T_f) \rightarrow \sigma_f^{-1}(\textrm{int}\hspace{.5mm}T_f) \] are both degree 2 coverings. Thus again there is a unique lift \[\hat\phi: \rho^{-2}(U) \rightarrow \sigma_f^{-2}(\textrm{int}\hspace{.5mm}T_f) \textrm{ of } \phi: \rho^{-1}(U) \rightarrow  \sigma_f^{-1}(\textrm{int}\hspace{.5mm}T_f)\]  such that $\hat\phi$ extends the definition of $\phi: \rho^{-1}(U) \rightarrow  \sigma_f^{-1}(\textrm{int}\hspace{.5mm}T_f)$. Defining $\phi(z):=\hat\phi(z)$ for $z\in \rho^{-2}(U)$, we have that: \[ \phi: D \cup \bigcup_{n=1}^2 \rho^{-n}(D) \rightarrow T_f \cup \bigcup_{n=1}^2 \sigma_f^{-n}(\textrm{int}\hspace{.5mm}T_f) \textrm{ is conformal.}\]

Since $c_f\not\in I(\sigma_f)$, we can repeat this procedure to obtain a conformal map: \[ \phi: D \cup \bigcup_{n=1}^\infty \rho^{-n}(D) \rightarrow T_f \cup \bigcup_{n=1}^\infty \sigma_f^{-n}(\textrm{int}\hspace{.5mm}T_f). \] Since \[ D \cup \bigcup_{n=1}^\infty \rho^{-n}(D)=\mathbb{H} \textrm{ and } T_f \cup \bigcup_{n=1}^\infty \sigma_f^{-n}(\textrm{int}\hspace{.5mm}T_f) = I(\sigma_f)  \] by Proposition \ref{fundamental_domain} and Definition \ref{tile_escaping}, respectively, the proof is finished.



\end{proof}


\section{Conformal Mating Definition}\label{conf_mating_def_sec}

We now define the notion of conformal mating in Theorem \ref{main_theorem}. Our definitions follow \cite{AFST_2012_6_21_S5_839_0} loosely, to which we refer for a more extensive discussion of conformal mating in the polynomial setting. 

\begin{rem}\label{mark_part_rem} Let $\hat{\mathbb{R}}:=\mathbb{R}\cup\{\infty\}$. The map $\rho: \hat{\mathbb{R}}\rightarrow \hat{\mathbb{R}}$ admits a Markov partition with three partition pieces: \begin{enumerate} \item $I_1:=\{\infty\}\cup(-\infty,-1]$, \item $I_2:=[-1,1]$ and \item $I_3:=[1,\infty)\cup\{\infty\}$,\end{enumerate} such that: \begin{enumerate} \item $\rho(I_1)\subset I_1\cup I_2$, \item $\rho(I_2)\subset I_1\cup I_3$ and  \item $\rho(I_3)\subset I_2\cup I_3$. \end{enumerate} The map $z\mapsto z^2$ on $\mathbb{T}$ also admits a Markov partition with intervals at endpoints $0$, $e^{2\pi i/3}$, $e^{4\pi i/3}$ (see Figure \ref{fig:Conjugacy}). It is readily checked that this Markov partition yields the same transition matrix as for $\rho: \hat{\mathbb{R}}\rightarrow \hat{\mathbb{R}}$ given above. Thus one can define a homeomorphism $\mathcal{E}: \hat{\mathbb{R}}\rightarrow \mathbb{T}$ such that $\mathcal{E}(\infty)=1$, $\mathcal{E}(1)=e^{2\pi i/3}$, $\mathcal{E}(-1)=e^{4\pi i/3}$, and \[ \mathcal{E}^{-1} \circ (z\mapsto z^2) \circ \mathcal{E} \equiv \rho. \]
\end{rem}

\begin{figure}
\centering
\scalebox{.5}{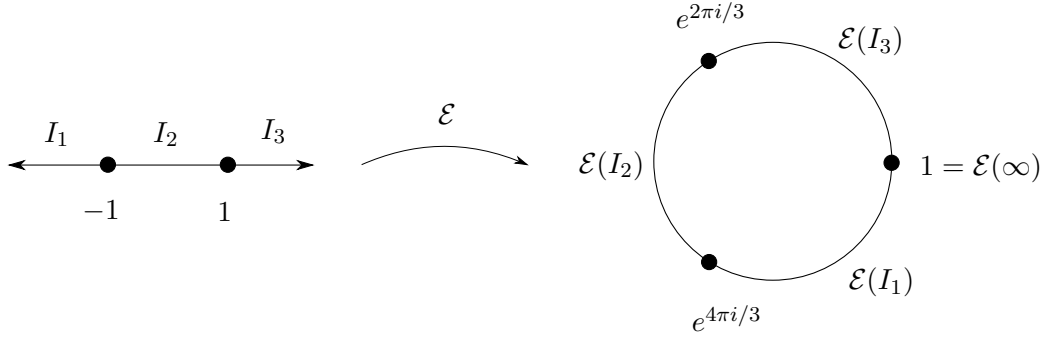}
\captionsetup{width=.9\textwidth}
\caption{Illustrated is the conjugacy $\mathcal{E}:\hat{\mathbb{R}}\rightarrow \mathbb{T}$ between $\rho: \hat{\mathbb{R}}\rightarrow \hat{\mathbb{R}}$  and the map $z\mapsto z^2$ on $\mathbb{T}$. Also marked are the relevant Markov partitions.  }
\label{fig:Conjugacy}
\end{figure}


\begin{rem}\label{setting_up_equiv_reltn} Let $p\in\textrm{Pol}_2$ such that $\mathcal{J}(p)$ is connected and locally connected. Denote by $\phi_p: \mathbb{D}^*\rightarrow \mathcal{B}_\infty(p)$ the B\"ottcher coordinate for $p$ such that $\phi_p'(\infty)=1$. We note that since $\partial K(p)=\mathcal{J}(p)$ is locally connected by assumption, it follows that $\phi_p$ extends to a semi-conjugacy between $z\mapsto z^{2}|_{\mathbb{T}}$ and $p|_{\mathcal{J}(p)}$.
 \end{rem}

\begin{definition}\label{conf_mating_equiv_reltn} 
Let notation be as in Remark~\ref{setting_up_equiv_reltn}. We define an equivalence relation $\sim$ on $\overline{\mathbb{H}} \sqcup K(p)$ by specifying $\sim$ is generated by $t\sim\phi_p\circ\mathcal{E}(t)$ for all $t\in\hat{\mathbb{R}}$.
\end{definition}

\begin{definition}\label{mating} Let $p\in\textrm{Pol}_2$ such that $\mathcal{J}(p)$ is connected and locally connected, and $f\in\mathcal{R}$. We say that $\sigma_f$ is a \emph{conformal mating} of $\Gamma$ with $p$ if there exist continuous mappings \[ \psi_p: K(p) \rightarrow K(\sigma_f) \textrm{ and } \psi_\Gamma: \overline{\mathbb{H}} \rightarrow \overline{I(\sigma_f)}  \] conformal on $\mathcal{F}(p)$, $\mathbb{H}$, respectively, such that \begin{enumerate} \item $\psi_p\circ p(w)=\sigma_f\circ\psi_p(w)$ for $w\in K(p)$, \item$\psi_\Gamma(U)=T_f$ with $\psi_\Gamma\circ \rho(z)=\sigma_f\circ\psi_\Gamma(z)$ for $z\in \overline{\mathbb{H}}\setminus U$, \item $\psi_\Gamma(z)=\psi_p(w) \textrm{ if and only if } z\sim w$ where $\sim$ is as in Definition \ref{conf_mating_equiv_reltn}. \end{enumerate}
\end{definition}

\section{Topology of the Filled Julia set}\label{topology}

\begin{notation} We will denote \[ f_0(z):=\frac{z}{1+z^3/2}, \] and $\sigma_0:=\sigma_{f_0}$. We let $A_0$ denote the basin of attraction of $0$ for $\sigma_0$.
\end{notation}

\begin{rem}\label{rel_with_LLMM1} One readily verifies the identity \begin{equation}\label{base_point} \frac{1}{f_0(1/z)}=z+\frac{1}{2z^2}. \end{equation} Let $\sigma_\textrm{inv}$ denote the Schwarz function of (\ref{base_point}). The Schwarz-\emph{reflection} map of (\ref{base_point}) is defined exactly as in Definition \ref{sigma_definition_larger}, but with the map $z\mapsto 1/\overline{z}$ replacing the map $M$. It is straightforward to verify then that the Schwarz-reflection map of (\ref{base_point}) is \begin{equation}\label{schwarz_reflection}z\mapsto\overline{\sigma_{\textrm{inv}}(z)}. \end{equation} The dynamics of (\ref{schwarz_reflection}) were studied in detail in \cite{LLMM1}. The proof of Theorem \ref{local_connectivity} below follows directly from results of \cite{LLMM1}. The rest of the results in this Section can also be deduced directly from results in \cite{LLMM1}, but we include proofs (mimicking those of \cite{LLMM1}) for the sake of the reader.  


\end{rem}


\begin{prop}\label{connectivity_of_basin} $A_0$ is connected.
\end{prop}

\begin{proof} Suppose by way of contradiction that $A_0$ contains more than one component. It must be then that  $A_0$ contains a Fatou component $U\not=A_0$ such that $\sigma(U)=A_0$. But $0$ is a preimage of $0$ (under $\sigma$) of multiplicity two, and $0$ has only two preimages (counted with multiplicity) under $\sigma_0$ by Proposition \ref{first_covering_property}. Thus we have reached a contradiction. 





\end{proof}

\begin{thm}\label{local_connectivity} $\partial I(\sigma_0)$ is locally connected.
\end{thm}

\begin{proof} Consider the Schwarz reflection map $\overline{\sigma_{\textrm{inv}}}$ as in Remark \ref{rel_with_LLMM1}. We have the relation \[ I(\overline{\sigma_{\textrm{inv}}}) = I(\sigma_{\textrm{inv}}). \] Thus Lemma 5.12 of \cite{LLMM1} applies directly to show that $\partial I(\overline{\sigma_{\textrm{inv}}})$, and hence $\partial I(\sigma_{\textrm{inv}})$, is locally connected. Since $\partial I(\sigma_0)$ is the image of $\partial I(\sigma_{\textrm{inv}})$ under $z\mapsto 1/z$, the Theorem follows.


\end{proof}



\begin{prop}\label{common_boundary} $ \partial I(\sigma_0)=  \partial A_0 $.
\end{prop}

\begin{proof} Let $z\in \partial A_0$. Suppose by way of contradiction that $z\not \in \partial I(\sigma_0)$. Then all iterates $\sigma_0^n$ are defined, and normal, in some neighborhood $U$ of $z$. The family $\sigma^n_0|_{U\cap A_0}$ converges uniformly on compact subsets to the constant $0$, and so the same holds true of $\sigma^n_0|_{U}$. Thus $z\in U\subset A_0$, and this is a contradiction.

We now prove that $ \partial I(\sigma_0) \subset  \partial A_0$. Let $0<x<1$, and note that: \[ \sigma_0\left( \frac{x}{1+x^3/2} \right) = \frac{x^2}{x^3+1/2} < \frac{x}{1+x^3/2}  \] where the inequality is easily checked. It follows that $f((0,1))\subset A_0$. Thus $f(1) \in \partial I(\sigma_0)$. Let $\phi: \mathbb{H} \rightarrow I_0$ denote the conformal map of Theorem \ref{group_conjugacy}. By Theorem \ref{local_connectivity} and Carath\'eodory's Theorem (see Theorem 17.14 of \cite{MR2193309}), $\phi$ extends continuously to a surjective map $\phi: \hat{\mathbb{R}} \rightarrow \partial I_0$.  One has $\phi(\infty)=f(1)$ by definition of $\phi$, and $\rho(x) = \phi^{-1}\circ\sigma_f\circ\phi(x)$ for $x\in\mathbb{R}$. Thus if $x\in\mathbb{R}$ is such that $\rho^n(x)=\infty$ for $n>0$, then $\sigma_f^n\circ\phi(x)=f(1)$, and so we also have $\phi(x)\in \partial A_0$. As $\{ x\in\mathbb{R} : \rho^n(x)=\infty \textrm{ for some } n>0\}$ is dense in $\mathbb{R}$,  it follows that $ \partial I(\sigma_0) \subset  \partial A_0$.
\end{proof}


\begin{prop}\label{Jordan_curve} $\partial I(\sigma_0)$ is a Jordan curve.
\end{prop}

\begin{proof} By Theorem \ref{local_connectivity}, $\partial I(\sigma_0)$ is locally connected. Thus any Riemann map $\phi: \hat{\mathbb{C}}\setminus\mathbb{D} \rightarrow I(\sigma_0)$  extends continuously to a surjective map $\phi: \mathbb{T} \rightarrow \partial I(\sigma_0)$ by Carath\'eodory's Theorem. Suppose by way of contradiction that $\phi(z_1)=\phi(z_2)$ for $z_1\not=z_2$ where $z_1, z_2\in\mathbb{T}$. Then $\phi(z_1)=\phi(z_2)$ is a cut-point of $I(\sigma_0)$, and the curve \begin{equation}\label{separation} \{\phi(tz_1) : t\geq1\} \cup  \{\phi(tz_2) : t\geq1\} \cup \{ \phi(z_1) \}\end{equation} separates $\partial I(\sigma_0)\setminus\{\phi(z_1)\}$. If $z\in\partial I(\sigma_0)\setminus\{\phi(z_1)\}$, then any neighborhood of $z$ intersects $A_0$ by Proposition \ref{common_boundary}. Thus (\ref{separation}) also constitutes a separation of $A_0$, which contradicts connectivity of $A_0$ (Proposition \ref{connectivity_of_basin}).
\end{proof}

\begin{cor}\label{summary} $\mathcal{J}(\sigma_0)$ is a Jordan curve, and $\mathcal{J}(\sigma_0) = \partial I(\sigma_0) = \partial A_0 $.
\end{cor}

\begin{proof} Since $\mathcal{J}(\sigma_0):=\partial K(\sigma_0)$ and $ K(\sigma_0)= \hat{\mathbb{C}}\setminus I(\sigma_0)$, it follows that $\mathcal{J}(\sigma_0)=\partial I(\sigma_0)$ and so $\mathcal{J}(\sigma_0)$ is a Jordan curve by Proposition \ref{Jordan_curve}. That $\partial I(\sigma_0) = \partial A_0 $ is the statement of Proposition \ref{common_boundary}.
\end{proof}




\section{Quasiconformal Surgery}\label{qc_surgery_sec}

\begin{rem}\label{Blaschke_product} In this remark, we summarize the discussion in Section 4.2.1 of \cite{BF14}, to which we refer for more details. We let \[ B_\lambda(z):= z\frac{z+\lambda}{1+\overline{\lambda}z}\textrm{, where }\lambda\in\mathbb{D}. \]

\noindent Let $\lambda\in\mathbb{D}\setminus\{0\}$, and choose $r>|\lambda|$. Let $h_\lambda$ be a quasiregular interpolation defined on $\mathbb{D}_{\sqrt{r}} \setminus \mathbb{D}_r$ between the degree 2 maps $B_\lambda|_{|z|=r}$ and $z\mapsto z^2|_{|z|=\sqrt{r}}$. We define \begin{align}\label{quasiregular_def} 
g_\lambda(z) := \left\{
        \begin{array}{ll}
        z^2 & \quad \textrm{ if } z \in \mathbb{D}\setminus \mathbb{D}_{\sqrt{r}}, \\            
            h_\lambda(z) & \quad \textrm{ if } z \in \mathbb{D}_{\sqrt{r}} \setminus \mathbb{D}_r, \\
            B_\lambda(z) & \quad \textrm{ if } z \in \mathbb{D}_r.
        \end{array}
    \right.
 \end{align} We define $\mu_\lambda\in L^\infty(\mathbb{D}_{\sqrt{r}})$ by $\mu_\lambda(z)=0$ for $z\in\mathbb{D}_r$ and $\mu_\lambda\equiv(h_\lambda)_{\overline{z}}/(h_\lambda)_{z}$ in $\mathbb{D}_{\sqrt{r}} \setminus \mathbb{D}_r$. Lastly, we extend $\mu_\lambda$ to $\mathbb{D}$ by pulling back $\mu_\lambda|_{\mathbb{D}_{\sqrt{r}}}$ under the map $z\mapsto z^2$ on $\mathbb{D}\setminus \mathbb{D}_{\sqrt{r}}$. Thus $\mu_\lambda\in L^\infty(\mathbb{D})$, and it is readily verified that $\mu_\lambda$ is $g_\lambda$-invariant. 
 
\end{rem}
 
\begin{prop} Let $\lambda\in\mathbb{D}$ and $g_\lambda$, $\mu_\lambda$, $B_\lambda$ as in Remark \ref{Blaschke_product}. Then there exists a quasiconformal integrating map $\phi_\lambda: \mathbb{D}\rightarrow \mathbb{D}$ for $\mu_\lambda$ such that $\phi_\lambda \circ g_\lambda \circ \phi_\lambda^{-1} = B_\lambda$. 
\end{prop} 

\begin{proof} Let $\phi_\lambda: \mathbb{D} \rightarrow \mathbb{D}$ be an integrating map for $\mu_\lambda$ such that $\phi_\lambda(0)=0$ and $\phi_\lambda(1)=(1-\lambda)/(1-\overline{\lambda})$. Then $\phi_\lambda \circ g_\lambda \circ \phi_\lambda^{-1}$ is a Blaschke product, and since $\phi_\lambda \circ g_\lambda \circ \phi_\lambda^{-1}$ fixes $(1-\lambda)/(1-\overline{\lambda})$, the conclusion follows.
\end{proof}

\begin{definition} Let $\Phi: A_0 \rightarrow \mathbb{D}$ be the B\"ottcher coordinate of Theorem \ref{bottcher_coord} applied to $\sigma_0|_{A_0}$. Define $\nu_\lambda$ to be the pullback (under $\phi$) of $\mu_\lambda$:  \[\nu_\lambda(z):=\mu_\lambda(\Phi(z))\frac{\overline{\partial_z f(z)}}{\partial_z f(z)} \textrm{ for } z\in A_0,\] and $\nu_\lambda(z)=0$ for $z\not\in A_0$. Thus $||\nu_\lambda||_{L^\infty(\mathbb{C})}<1$.
\end{definition}

\begin{rem}\label{symmetry_remark} Suppose $\lambda\in\mathbb{D}\cap\mathbb{R}$. Then the interpolation $h_\lambda$ of Remark \ref{Blaschke_product} can be chosen to satisfy $h(\overline{z})=\overline{h(z)}$, and hence the Beltrami coefficient $\mu_\lambda$ will satisfy the same relation. The B\"ottcher coordinate $\Phi: A_0 \rightarrow \mathbb{D}$ of $\sigma_0|_{A_0}$ can also be chosen to satisfy $\Phi(\overline{z})=\overline{\Phi(z)}$, and hence $\nu_\lambda$ will satisfy the same relation. Thus if we denote by $\psi_\lambda: \mathbb{C} \rightarrow\mathbb{C}$ an  integrating map for $\nu_\lambda$, we may ensure that \begin{equation}\label{symmetry} \psi_\lambda(\overline{z})=\overline{\psi_\lambda(z)}. \end{equation}
\end{rem}

 

\begin{thm}\label{qd_domain_surgery} Let $\lambda\in(-1,1)$, and $\psi_\lambda:\mathbb{C}\rightarrow\mathbb{C}$ an integrating map for $\nu_\lambda$ satisfying \emph{(\ref{symmetry})}. Then $\psi_\lambda\circ f_0(\mathbb{D})$ is a quadrature domain. 
\end{thm}

\begin{proof} We define the function \begin{align}
 F_\lambda(z) := \left\{
        \begin{array}{ll}
        \sigma_0(z) & \quad \textrm{ if } z \in f_0(\mathbb{D})\setminus A_0, \\            
            \Phi^{-1} \circ g_\lambda \circ \Phi(z) & \quad \textrm{ if } z \in A_0.
        \end{array}
    \right.
     \end{align}


\noindent We claim that $F_\lambda$ is quasiregular. Indeed, $\sigma_0$ and $\Phi^{-1} \circ g_\lambda \circ \Phi$ are both quasiregular in $f_0(\mathbb{D})\setminus A_0$, $A_0$, respectively. Moreover, $\sigma_0$ and $F_\lambda$ agree on a neighborhood of $\partial A_0$ (recall $\partial A_0$ is locally connected by Theorem \ref{local_connectivity}) by definition of $g_\lambda$. Thus $F_\lambda$ is quasiregular. Observing that $\nu_\lambda$ is $F_\lambda$-invariant (since $\mu_\lambda$ is $g_\lambda$-invariant), we see that $\psi_\lambda \circ F_\lambda \circ \psi_\lambda^{-1}$ is holomorphic in $\psi_\lambda\circ f_0(\mathbb{D})$. Let \[\sigma_\lambda(z):=\psi_\lambda \circ F_\lambda \circ \psi_\lambda^{-1}(z) \textrm{ for } z\in\psi_\lambda\circ f_0(\mathbb{D}).\] By (\ref{symmetry}), we have the relation: \[ \sigma_\lambda(z) = \overline{z} \textrm{ for } z\in\psi_\lambda\circ f_0(\mathbb{T}). \] Thus by Lemma 2.3 of \cite{AS}, the domain $\psi_\lambda\circ f_0(\mathbb{D})$ is a quadrature domain and hence has a Schwarz function $\sigma$. Since $\sigma=\sigma_\lambda$ on $\psi_\lambda\circ f_0(\mathbb{D})$, we have $\sigma=\sigma_\lambda$.

\end{proof}

\begin{notation} Let $\psi_\lambda$ be as in Theorem \ref{qd_domain_surgery}. As in the proof of Theorem \ref{qd_domain_surgery}, we will denote the Schwarz function of $\psi_\lambda\circ f_0(\mathbb{D})$ by $\sigma_\lambda$.
\end{notation}
 
\begin{cor}\label{cor_of_qd_surgery} Let $\lambda\in(-1,1)$. Then $\psi_\lambda(0)$ is a fixed point of $\sigma_\lambda$ with multiplier $\lambda$. 
\end{cor}

\begin{proof} We continue with notation as in the proof of Theorem \ref{qd_domain_surgery}. Unravelling the definitions, we have \begin{equation}\label{fixed_point_multiplier} \sigma_\lambda\circ\psi_\lambda(0) = \psi_\lambda \circ  \Phi^{-1} \circ B_\lambda \circ \Phi \circ \psi_\lambda^{-1} \circ\psi_\lambda(0). \end{equation} From (\ref{fixed_point_multiplier}) it is evident that $\psi_\lambda(0)$ is fixed under $\sigma_\lambda$. Noting that $\psi_\lambda$ and $\Phi$ are both conformal in a neighborhood of $0$, it follows that $\sigma_\lambda'(\psi_\lambda(0))=\lambda$ since $B_\lambda'(0)=\lambda$. 


\end{proof}



\begin{proof}[Proof of Theorem~\ref{main_theorem}]  Let $c\in(-.75, .25)$. Then $p_c$ has a fixed point of multiplier $\lambda_c\in(-1,1)$. Let $\sigma:=\sigma_{\lambda_c}$ be the Schwarz function of the quadrature domain $\psi_{\lambda_c}\circ f_0(\mathbb{D})$ as in Theorem \ref{qd_domain_surgery}. We will show that $\sigma$ is a conformal mating of $p_c$ and $\Gamma$, namely we will verify Conditions (1)-(3) in Definition \ref{mating}, starting with Condition (1).

\vspace{5mm}

By Corollary \ref{cor_of_qd_surgery}, $\sigma$ has a fixed point of multiplier $\lambda_c$. Then by Theorem \ref{Fatou_set_conjugacy}, there is a conformal map $\psi_p: \mathcal{F}(p) \rightarrow \mathcal{F}(\sigma)$ satisfying \begin{equation}\label{polynomial_conjugacy} \psi_p\circ p(z)=\sigma_f\circ\psi_p(z) \textrm{ for } z\in\mathcal{F}(p). \end{equation} Note $\mathcal{J}(\sigma)$ is a Jordan curve since it is a quasiconformal image of the Jordan curve $\partial I(\sigma_0)$ (see Proposition \ref{Jordan_curve}). Thus since $\mathcal{J}(p)$ is a Jordan curve, the relation (\ref{polynomial_conjugacy}) extends to $\psi_p: \mathcal{J}(p) \rightarrow \mathcal{J}(\sigma)$. Thus we have verified Condition (1).


\vspace{5mm}

Next we consider Condition (2) of Definition \ref{mating}. As $\psi_\lambda\circ f_0(\mathbb{D})$ is a $\mathbb{R}$-symmetric quadrature domain, it is readily checked that there is a regular function (see Definition \ref{regular}) $f$ such that $\sigma=\sigma_f$. Let $\psi_\Gamma: \mathbb{H} \rightarrow I(\sigma)$ be the map obtained by applying Theorem \ref{group_conjugacy} to the regular map $f$. Then by Theorem \ref{group_conjugacy}, we have that $\psi_\Gamma(\overline{U})=\overline{T_f}$ and \begin{equation}\label{conjugacy_inside} \psi_\Gamma\circ \rho(z)=\sigma_f\circ\psi_\Gamma(z) \textrm{ for } z\in \mathbb{H}\setminus U. \end{equation} Since $\partial I(\sigma)=\mathcal{J}(\sigma)$ is a Jordan curve by Proposition \ref{Jordan_curve}, the conformal map $\psi_\Gamma$ extends to a homeomorphism $\psi_\Gamma: \hat{\mathbb{R}} \rightarrow \mathcal{J}(f)$ satisfying (\ref{conjugacy_inside}) for $z\in\hat{\mathbb{R}}$. 


\vspace{5mm}

Lastly, we verify Condition (3) of Definition \ref{mating}. We first show that \begin{equation}\label{demonstrating_equiv} \psi_\Gamma(t)=\psi_p(\phi_p\circ\mathcal{E}(t)) \textrm{ for all } t\in\hat{\mathbb{R}}. \end{equation} We record the relations \begin{align}\label{record_relations} \mathcal{E} \circ \psi_\Gamma^{-1} \circ \sigma(w) = (z\mapsto z^2) \circ \mathcal{E} \circ \psi_\Gamma^{-1}(w) \textrm{ for } w\in \mathcal{J}(\sigma)\textrm{, and } \\ \label{record_relations2} \psi_p \circ \phi_p (z^2) = \sigma \circ \psi_p \circ \phi_p(z) \textrm{ for } z\in\mathbb{T}.    \end{align} Using (\ref{record_relations}) and (\ref{record_relations2}) we compute \begin{equation} \mathcal{E} \circ \psi_\Gamma^{-1} \circ \psi_p \circ \phi_p (z^2) = (z\mapsto z^2) \circ  \mathcal{E} \circ \psi_\Gamma^{-1} \circ \psi_p \circ \phi_p(z) \textrm{ for } z\in\mathbb{T}. \end{equation} In other words, \begin{equation}\label{identity} \mathcal{E} \circ \psi_\Gamma^{-1} \circ \psi_p \circ \phi_p: \mathbb{T} \rightarrow \mathbb{T} \end{equation} is an orientation-preserving homeomorphism of $\mathbb{T}$, which conjugates $z\mapsto z^2$ to itself. The only such map is the identity, and so we have: \begin{equation}\label{identity2}\mathcal{E} \circ \psi_\Gamma^{-1} \circ \psi_p \circ \phi_p(s)=s \textrm{ for } s\in\mathbb{T}.  \end{equation} Thus, the relation (\ref{demonstrating_equiv}) follows by taking $s=\mathcal{E}(t)$ in (\ref{identity2}).

It remains to show that if $\psi_\Gamma(z)=\psi_p(w)$, then $z\sim w$. To this end, let us assume that $z$, $w$ are such that $\psi_\Gamma(z)=\psi_p(w)$. Note then that this implies $z=t\in\hat{\mathbb{R}}$ and $w=\phi_p(s)$ for $s\in\mathbb{T}$. Thus $t=\psi_\Gamma^{-1}\circ\psi_p\circ\phi_p(s)$. By (\ref{identity2}), we also have \[ \psi_\Gamma^{-1}\circ\psi_p\circ\phi_p(s) = \mathcal{E}^{-1}(s). \] We conclude that $s=\mathcal{E}(t)$. Thus since $t\sim\phi_p\circ\mathcal{E}(t)$ by definition of $\sim$, and $z=t$ and $w=\phi_p(s)=\phi_p\circ\mathcal{E}(t)$, we conclude that $z\sim w$ as needed. 

\end{proof}

\section{Simultaneous Uniformization}\label{sim_unif_sec}





\begin{definition}\label{Fuchsian_group} Let $d\geq2$. We define a Fuchsian group $\Gamma_d$ as follows. Let $(C_j)_{j=1}^{d+1}$ be the Euclidean circles such that $C_j$ intersects $|z|=1$ at right-angles at the points $\exp{(2\pi i(j-1)/(d+1))}$, $\exp{(2\pi i j/(d+1))}$. Denote the common radius of the circles $(C_j)$ by $r$, and denote the center of $C_j$ by $z_j$. The group $\Gamma_d$ is defined by generators \begin{equation} \gamma_j(z):=\overline{z_j}+\frac{r^2}{z-z_j} \end{equation} for $1 \leq j \leq d/2+1$ or $1 \leq j \leq (d+1)/2$ according to whether $d$ is even or odd, respectively. 
\end{definition}

\begin{rem} When $d=2$, the Fuchsian group $\Gamma_2$ defined in Definition \ref{Fuchsian_group} is M\"obius conjugate to the group $\Gamma$ of Theorem \ref{main_theorem}.
\end{rem}

\begin{definition} Let $U$ denote the component of $\mathbb{D}\setminus\cup_j C_j$ containing $0$. We note that $\mathbb{D}\setminus U=\sqcup_j\textrm{int}\hspace{.5mm}C_j$. We define a map $\rho: \mathbb{D}\setminus U\rightarrow\mathbb{D}$ by $\rho(z):=\overline{z_j}+r^2/(z-z_j)$ if $z\in\textrm{int}\hspace{.5mm}C_j$. 


\end{definition}

\begin{figure}
\centering
\scalebox{.3}{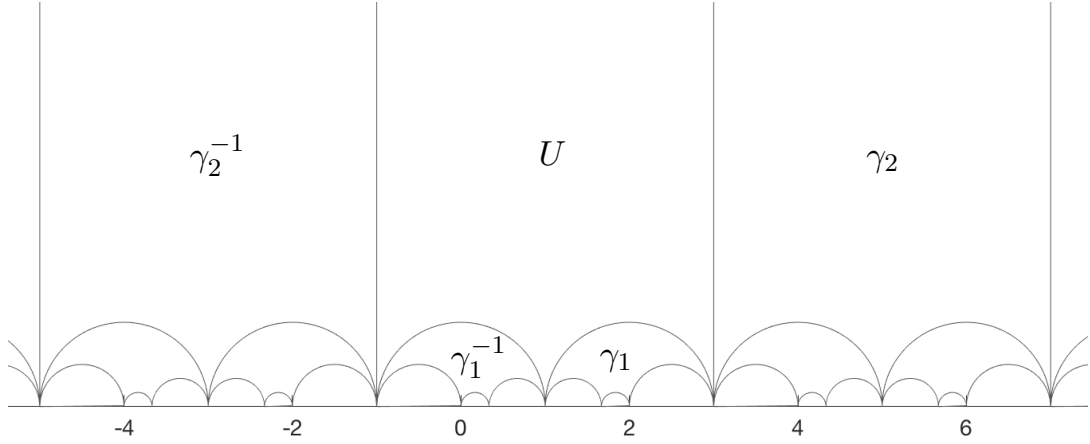}
\captionsetup{width=.9\textwidth}
\caption{Illustrated is a fundamental domain $U$ of the group $\Gamma_3$, together with several images under $U$ of elements of $\Gamma_3$. The group $\Gamma_3$ is shown acting on $\mathbb{H}$ rather than on $\mathbb{D}$ (as in Definition \ref{Fuchsian_group}) in order to be consistent with the convention for $\Gamma\equiv\Gamma_2$. Labels are affixed to those ``tiles'' arising from words of length 1. }
\label{fig:fund_dom_next.pdf_tex}
\end{figure}

\begin{rem} Definition \ref{mating} of conformal mating is restricted to the class $\mathcal{R}$, but the definition is readily extended to maps of higher degree, or one which is quasiconformally conjugate to a map in class $\mathcal{R}$. This is the setting in which we will work below.
\end{rem}

\noindent \emph{Sketch of Proof of Theorem \ref{Theorem_B}.}  We fix $p$, $\beta$ as in the statement of Theorem \ref{Theorem_B}. Let \begin{equation} f(z):=\frac{z}{1+z^{d+1}/d}. \end{equation} Then $f$ is univalent in $\mathbb{D}$, and $\sigma:=\sigma_f$ has a critical fixed point at $0$ of degree $d$. Let $A_\sigma$ denote the basin of attraction of $0$ for $\sigma$. There is a B\"ottcher coordinate $\Phi: A_\sigma \rightarrow \mathbb{D}$ conjugating $\sigma: A_\sigma\rightarrow A_\sigma$ to the map $z\mapsto z^d$ on $\mathbb{D}$. As $p$ is in the principal hyperbolic component in $\textrm{Pol}_d$, $p$ has a single attracting fixed point, and all critical points of $p$ are in the corresponding basin of attraction $A_p$.  Thus letting $\phi_p: A_p \rightarrow \mathbb{D}$ denote a conformal map, we have that $\phi_p\circ p\circ\phi_p^{-1}: \mathbb{D} \rightarrow \mathbb{D}$ is an expanding Blaschke-product of degree $d$. Denote this Blaschke product by $B$. By choosing an appropriate $r<1$, we can mimic the definition of (\ref{quasiregular_def}) and define a quasiregular interpolation $h$ between the degree $d$ maps $z\mapsto z^d$ on $|z|=\sqrt{r}$ and $B$ on $|z|=r$, so that: \begin{align} 
g(z) := \left\{
        \begin{array}{ll}
        z^d & \quad \textrm{ if } z \in \mathbb{D}\setminus \mathbb{D}_{\sqrt{r}}, \\            
            h(z) & \quad \textrm{ if } z \in \mathbb{D}_{\sqrt{r}} \setminus \mathbb{D}_r, \\
            B(z) & \quad \textrm{ if } z \in \mathbb{D}_r.
        \end{array}
    \right.
 \end{align} is a degree $d$ quasiregular map of $\mathbb{D}$. As in Section \ref{qc_surgery_sec}, we may define a $g$-invariant Beltrami coefficient $\mu$ on $\mathbb{D}$ by letting $\mu\equiv0$ in $\mathbb{D}_r$, $\mu\equiv h_{\overline{z}}/h_z$ in $\mathbb{D}_{\sqrt{r}} \setminus \mathbb{D}_r$, and pulling back elsewhere. Next we define \begin{align}
 F(z) := \left\{
        \begin{array}{ll}
        \sigma(z) & \quad \textrm{ if } z \in f_0(\mathbb{D})\setminus A, \\            
            \Phi^{-1} \circ g \circ \Phi(z) & \quad \textrm{ if } z \in A.
        \end{array}
    \right.
     \end{align} Let $\nu$ be the Beltrami coefficient defined on $A$ by pulling back $\mu$ under $\Phi$.


Now we consider the group structure of $\sigma$ on the escaping set $I=I_\sigma$ (see Definition \ref{tile_escaping}). By mimicking the proof of Theorem \ref{group_conjugacy}, we can show there exists a $\Gamma_d$-invariant conformal map $\phi_\Gamma: \mathbb{D} \rightarrow I$, such that $\phi_\Gamma(\overline{U})=\overline{T_f}$ and \begin{equation} \rho(z) = \phi_\Gamma^{-1}\circ\sigma_f\circ\phi_\Gamma(z) \textrm{ for } z\in \mathbb{D}\setminus\overline{U}. \end{equation} Since $\beta\in\beta(\Gamma_d)$, there is a $\Gamma_d$-invariant quasiconformal map $\phi_\beta: \mathbb{D}\rightarrow\mathbb{D}$ inducing the element $\beta\in\beta(\Gamma_d)$. Let $\mu_\beta$ be the Beltrami coefficient of $\phi_\beta\circ\phi_\Gamma^{-1}$ defined in $I$. Extend $\nu$ (up to now defined only in $A$) to the escaping set $I$ by defining $\nu\equiv\mu_\beta$ in $I$. Lastly, we define $\nu(z)=0$ for $z\in\hat{\mathbb{C}}\setminus(A\cup I)$.

We have defined $\nu$ so that $\nu$ is $F$-invariant: this follows as $\phi_\Gamma$, $\phi_\beta$ are $\Gamma_d$-invariant, and $\mu$ is $g$-invariant. Let $\psi: \mathbb{C}\rightarrow\mathbb{C}$ denote a straightening map for $\nu$. Define \begin{equation} p_c\sqcup\beta(z):=\psi \circ F \circ \psi^{-1}(z) \textrm{ for } z\in\psi\circ f(\mathbb{D}).  \end{equation} Then $p_c\sqcup\beta$ is holomorphic since $\nu$ is $F$-invariant. One readily checks that $p_c\sqcup\beta: \psi(A) \rightarrow \psi(A)$ and $p: \textrm{int}\hspace{.5mm}K(p) \rightarrow \textrm{int}\hspace{.5mm}K(p)$ are conformally conjugate. Likewise, $\phi_\beta\circ\phi_\Gamma^{-1}\circ\psi^{-1}: \psi(I)\rightarrow\mathbb{D}\setminus \phi_\beta(U)$ defines a conjugacy between $p_c\sqcup\beta: \psi(I)\rightarrow\hat{\mathbb{C}}$ and $\phi_\beta\circ\rho\circ\phi_\beta^{-1}: \mathbb{D}\setminus \phi_\beta(U) \rightarrow \mathbb{D}$. For $f$ as in (\ref{base_point}), one may mimic the proofs in Section \ref{topology} (or Lemma 4.10 of \cite{lazebnik2020bers}) to show that $A$ and $I$ share a common boundary which is a Jordan curve. Thus the same holds for the common boundary between the filled Julia set and escaping set of $p_c\sqcup\beta$. The last verification in the definition of conformal mating thus proceeds as in the proof of Theorem \ref{main_theorem}.
\qed




\section{Remarks}\label{remarks_section}

We remark briefly on several questions arising from the present work. The first is that the phase transition of Figure \ref{fig:deltoid_flow.pdf_tex} is conjectural: we have neither formulated nor proven a precise statement. One such statement would be that any family of (appropriately normalized) $\mathbb{R}$-symmetric, simply connected quadrature domains whose Schwarz functions have attracting fixed points with multipliers $\rightarrow-1^{+}$ converge to a quadrature domain of connectivity 2: the cardioid and circle (up to a M\"obius transformation). Moreover, the Schwarz function of the limiting cardioid and circle should be a conformal mating of $z\mapsto z^2-.75$ and $\Gamma$.

Another approach to the proof of Theorem \ref{main_theorem} would be to formulate explicitly the uniformizing Riemann maps of the simply connected quadrature domains, and calculate the multiplier of the attracting fixed point. The difficulty would then be in proving univalence of the maps. Nevertheless, we record the formulas. Consider the maps \begin{equation}\label{parameter_family} f_t(z):=\frac{z+t}{1+z^3/2}. \end{equation} Numerical evidence suggests that there exists $t_0<0$ such that for $t\in[t_0,1)$, the map $f_t$ is univalent on the interior of the circle passing through the three finite critical points of $f_t$. Moreover, $t=t_0$ yields a Schwarz function with a parabolic fixed point, and as $t\rightarrow 1^-$, one has the limiting behavior as discussed above and shown in Figure \ref{fig:deltoid_flow.pdf_tex} (up to inversion). One can also give formulas for an analogous family of maps conformal in the fixed domain $\mathbb{D}$.

We have focused on the holomorphic setting in this work, however there is also interest in studying the (anti-holomorphic) Schwarz reflection maps associated to the family (\ref{parameter_family}). Let $C_t$ denote the circle passing through the three finite critical points of $f_t$. Numerical evidence indicates that there is a natural homeomorphism of (1) the $\mathbb{C}$-parameters $t$ for which the maps (\ref{parameter_family}) are univalent in the interior of $C_t$ with (2) the principal hyperbolic component (the main deltoid) in the parameter space of the family $z\mapsto\overline{z}^2+c$: both are parametrized by the multiplier of the relevant attracting fixed point. The boundary is of  particular interest: simple parabolics on the boundary of the main deltoid correspond to simply connected quadrature domains, whereas the three double parabolics correspond to the cardioid and circle. 




\bibliographystyle{alpha}

\end{document}